\documentclass[11pt,letterpaper]{amsart}
\usepackage{amsmath}
\usepackage{amsthm}
\usepackage{amssymb}
\usepackage{amsfonts,mathrsfs}
\usepackage{stmaryrd}
\usepackage{amsxtra}  
\usepackage{epsfig}
\usepackage{verbatim}
\usepackage[all]{xy}

\usepackage{dsfont}
\usepackage{color}
\usepackage{mathtools}
\usepackage{tikz}
\usepackage{tikz-cd}
\usetikzlibrary{fpu}
\usetikzlibrary{matrix}

\usepackage{graphicx} 

\theoremstyle{plain}
\newtheorem{theorem}{Theorem}[section]

\newtheorem{lemma}[theorem]{Lemma}

\theoremstyle{definition}

\newtheorem{remark}[theorem]{Remark}

\usepackage{bm}

\title{Unbounded operators and the uncertainty principle}
\author{friedrich haslinger }

\thanks{Supported by the Austrian Science Fund (FWF) project  P36884.}
\address{Fakult\"at  f\"ur Mathematik\\ Universit\"at Wien\\
Oskar-Morgenstern-Platz 1\\
A-1090 Wien, Austria}
\email{friedrich.haslinger@univie.ac.at}



\begin{document}
\parindent 0 pt
\maketitle
\begin{abstract} We study a variant of the uncertainty principle in terms of the annihilation and creation operator on generalized Segal Bargmann spaces, which are used for the FBI-Bargmann transform. In addtion, we compute the Berezin transform of these operators and indicate how to use spaces of entire functions in one variable to study the Szeg\H{o} kernel for hypersurfaces in $\mathbb C^2.$
    
\end{abstract}
\section{Introduction}
In the Segal-Bargmann space (also Fock space) $A^2(\mathbb C^n, e^{-|z|^2})$ of entire functions $f: \mathbb C^n \longrightarrow \mathbb C$ such that 
$$\|f\|^2 = \int_{\mathbb C^n} |f(z)|	^2 \, e^{-|z|^2} \, d\lambda (z) <\infty,$$
the differentiation operators $a_j(f)= \frac{\partial f}{\partial z_j}$ (annihilation) and the multiplication operators $a^*_j(f)= z_j f$  (creation) are unbounded, densely defined adjoint operators with 
$$\|f\|^2 \le \|a_j( f) \|^2 + \| a^*_j(f)\|^2$$
 for $f\in {\text{dom}}(a_j),$ which corresponds to the uncertainty principle, see \cite{folland/book}. 

The uncertainty principle in its general form states that if $A$ and $B$ are quantum observables (i.e. self-adjoint operators), the probability distributions  of $A$ and $B$ cannot both be concentrated near single points in any state $u$ such that
$$((AB-BA)u,u) \neq 0.$$
\begin{theorem} If $A$ and $B$ are self-adjoint operators on a Hilbert space $\mathcal H,$ then
$$\|(A-a)u \| \, \|(B-b)u \| \ge \frac{1}{2} \left | ((AB-BA)u,u) \right |$$
for all $u\in {\text{dom}}(AB) \cap {\text{dom}}(BA)$ and all $a,b \in \mathbb R.$
\end{theorem}

We apply this result to the position operators $X_j(u) = x_j u$ and the momentum operators $D_ju = -i\frac{\partial u}{\partial x_j}$ on $L^2(\mathbb R^n)$ and get
\begin{equation}\label{uncert1}
\| (X_j-a)u \| \, \|(D_j-b)u \| \ge \frac{1}{2} \| u \|^2.
\end{equation}

With $a=b=0,$ and the inequality $\alpha \beta \le \frac{1}{2}(\alpha^2+ \beta^2)$ we get the following variant of the uncertainty principle

\begin{equation}\label{sec: uncert2}
\|X_j(u) \|^2 + \|D_j(u)\|^2 \ge  \|u\|^2, \, u\in {\text{dom}}(X_j) \cap {\text{dom}}(D_j).
\end{equation}
\vskip 0.5 cm
Actually, \eqref{uncert1} and \eqref{sec: uncert2} are equivalent.

We will describe this inequality in the Segal-Bargmann space (Fock space)
 $A^2(\mathbb C^n, e^{-|z|^2}).$

Let $a_j(f) = \frac{\partial f}{\partial z_j},$ where we take
$$f\in {\text{dom}} (a_j) = \{ f\in  A^2(\mathbb C^n, e^{-|z|^2}) :  \frac{\partial f}{\partial z_j} \in A^2(\mathbb C^n, e^{-|z|^2})  \}.$$

$a_j$ (annihilation)  is a densely defined, unbounded operator, its adjoint $a_j^*$ (creation) is given by $a_j^*(f)=z_jf$ and we have 
$$\|a_j^*(f)\|^2= \|a_j(f)\|^2 + \|f\|^2,$$
hence ${\text{dom}} (a_j)={\text{dom}} (a_j^*),$ and

\begin{equation}\label{fock}
\| f\|^2 \le \|a_j(f)\|^2 + \|a_j^*(f)\|^2, 
\end{equation}
for $f\in {\text{dom}} (a_j).$
In another similar context, this inequality is the so-called basic estimate for the corresponding complex Laplacian operator, see \cite{hasbook}, \cite{haslinger}, \cite{Has21}, \cite{Has22}, \cite{sonhas}, \cite{HS22}.

 Here we use the Bargmann transform $\mathcal B : L^2(\mathbb R^n) \longrightarrow
 A^2(\mathbb C^n, e^{-|z|^2}),$ which is given by 
 $$ \mathcal B (f) (z)= \pi^{-n/4} \int_{\mathbb R^n} f(x) \exp \big [- z^2/2+ \sqrt{2}zx -x^2/2)\big ]\,dx, \, f\in L^2(\mathbb R^n),$$
 here $z^2 = \sum_{j=1}^n z_j^2$ and $xz = \sum_{j=1}^n x_jz_j.$
 
 The Bargmann transform is a unitary operator\index{unitary operator} between $L^2(\mathbb R^n)$ and
 $A^2(\mathbb C^n, e^{-|z|^2}).$
 
  Concerning the operators $a_j$ and $a_j^*$ we have
 $$\mathcal B^* a_j \mathcal B = \frac{1}{\sqrt{2}}(X_j + iD_j )\, \ {\text{and}} \ \, \mathcal B^*a_j^* \mathcal B=\frac{1}{\sqrt{2}} (X_j-iD_j),$$
 
  So we get
$$ \|X_j(u)\|^2= \frac{1}{2} \|\mathcal B^* a_j \mathcal B(u)+ \mathcal B^*a_j^* \mathcal B (u) \|^2,$$
and 
$$ \|D_j(u)\|^2= \frac{1}{2} \|\mathcal B^* a_j \mathcal B(u)- \mathcal B^*a_j^* \mathcal B (u) \|^2.$$
Hence 
\begin{eqnarray*}
\|X_j(u) \|^2 + \|D_j(u)\|^2 &=&  \|\mathcal B^* a_j \mathcal B(u) \|^2+
 \|\mathcal B^* a^*_j \mathcal B(u)\|^2\\
&=&\| a_j\mathcal B(u) \|^2 + \| a^*_j\mathcal B(u) \|^2\\
&\ge & \|\mathcal B(u) \|^2 \\
&=& \|u \|^2,
\end{eqnarray*}
hence \eqref{fock} and \eqref{sec: uncert2} are equivalent.

 On the Segal-Bargmann space we have the following commutation relations:
$$[a_j,a_j^*] = a_ja_j^* - a_j^*a_j =I,$$
$j=1, \dots,n.$ 

In this paper we study the operators $a_j$ and its adjoints $a_j^*$ in more general weighted spaces of entire functions which play a crucial role in semiclassical analysis and quantization (see \cite{RSN}, \cite{Z}). In sake of simplicity we will carry out all computations for the case $n=2,$ as the general case contains no further difficulties. 

Let 
$$A=\begin{pmatrix}\beta_1 & \kappa\\ \kappa & \beta_2 \end{pmatrix}, $$
with $\beta_1,\beta_2,\kappa \in \mathbb R.$ For $z= \binom{z_1}{z_2} \in \mathbb C^2$ define
$$\langle Az,z \rangle = \beta_1 z_1^2 + 2\kappa z_1 z_2 + \beta_2 z_2^2, $$
and let
$$H_{A} = \left \{ f \, {\text{entire}} : 
 \int_{\mathbb C^2} |f(z)|^2\,\exp (-\Re \langle Az,z \rangle -|z|^2)\, d\lambda (z) <\infty  \right \}.$$

 Using the canonical orthonormal basis of $H_A$ we investigate the densely defined operators
 $a_j(f) = \frac{\partial f}{\partial z_j}$ and $b_j(f)= z_jf$ on $H_A,$ and establish analogous estimates to the uncertainty principle \eqref{fock}.
\vskip 0.5 cm
\section{The FBI-Bargmann transform}
By means of the FBI-Bargmann transform one obtains a correspondence between Weyl operators acting on $L^2(\mathbb R^n)$ and Weyl operators acting on $H_A,$ see \cite{RSN}. In \cite{ACDKS} a similar approach is used for some other interesting applications. 
We will use the canonical orthonormal basis in $H_A$ to compute the adjoint operators. We apply Parseval's identity, and it will be important to know that convergence in the norm of $H_A$ implies convergence of entire functions in the compact open topology, which follows from

\begin{lemma}\label{lemm1} Let  $\varphi : \mathbb C^n \longrightarrow  \mathbb R$ be a real valued continuous function and define
$$A^2(\mathbb C^n, e^{-\varphi }) = \{ f \, {\text{entire}} \, : 
 \| f \|^2= \int_{\mathbb C^n} |f(z)|^2 e^{-\varphi (z)}\, d\lambda (z) <\infty \}.$$
Then, for each compact subset $K\subset \mathbb C^n,$ there exists a constant $C_K>0$ such that
\begin{equation}\label{compconv}
\sup_{z\in K}|f(z)| \le C_K \|f\|,
\end{equation} 
for each $f\in A^2(\mathbb C^n, e^{-\varphi }).$
\end{lemma}
\begin{proof} We consider only the case $n=1,$ as iterations of Cauchy's theorem imply the general result: 
\begin{eqnarray*}
|f(z)| &\le & \frac{1}{\pi r^2} \int_{D(z,r)} e^{\varphi (w)/2} |f(w)| e^{-\varphi (w)/2} \,d\lambda (w)\\
& \le & \frac{1}{\pi r^2} \left ( \int_{D(z,r)} e^{\varphi (w)} \,d\lambda (w) \right )^{1/2} \, \| f\|.
\end{eqnarray*}

\end{proof}
Next we indicate that the spaces $H_A$ and $A^2(\mathbb C^n, e^{-|z|^2})$ are isometric isomorphic. The isometry 
$$\Phi : A^2(\mathbb C^n, e^{-|z|^2}) \longrightarrow H_A$$

is given by $\Phi (f) (z) = f(z)  \exp \left ( \frac{1}{2} \langle Az,z \rangle \right ),$ for $f\in A^2(\mathbb C^n, e^{-|z|^2}).$ 

 The functions 
 $$\phi_\alpha (z)=\frac{z_1^{\alpha_1}}{\sqrt{\pi \alpha_1 !}}
\, \frac{z_2^{\alpha_2}}{\sqrt{\pi \alpha_2 !}},$$
$\alpha_1, \alpha_2 =0,1,2,\dots.$
form the canonical orthonormal basis for $A^2(\mathbb C^2, e^{-|z|^2}).$
 Then the functions
 $$\psi_\alpha (z) = \frac{z_1^{\alpha_1}}{\sqrt{\pi \alpha_1 !}}
\, \frac{z_2^{\alpha_2}}{\sqrt{\pi \alpha_2 !}}\, \exp \left ( \frac{1}{2} \langle Az,z \rangle \right ),$$
constitute an orthonormal basis in $H_A:$ for the inner product in $H_A$ we have
$$(\psi_\alpha, \psi_{\alpha'}) = \int_{\mathbb C^2}
\psi_\alpha (z)
\left( \psi_{\alpha'} (z)\right )^- \exp (-\Re \langle Az,z \rangle -|z|^2)\, d\lambda (z)= (\phi_\alpha, \phi_{\alpha'}),$$
where the last inner product is the inner product of the Segal Bargmann space.

If $f\in H_A,$ we write $f= \sum_\alpha f_\alpha \psi_\alpha,$ where $(f_\alpha)_\alpha \in l^2.$ Define
$${\text{dom}}(a_j) = \{ f\in H_A : a_j(f) \in H_A \}, \, j=1,2.$$
Similarly, let $b_j(f) =z_jf$ and
$${\text{dom}}(b_j) = \{ f\in H_A : b_j(f) \in H_A \}, \, j=1,2.$$
The operators $a_j, b_j$ are densely defined unbounded operators on $H_A$ with closed graph, which implies that the domains of $a_j^*$ and $b_j^*$ are also dense in $H_A;$ (see \cite{hasbook}).

We compute the images of the basis elements under $a_j$ and $b_j.$
We have
\begin{equation}\label{b_j}
b_1(\psi_\alpha ) = \sqrt{\alpha_1 +1}\,  \psi_{\alpha_1+1, \alpha_2} \, , \, b_2(\psi_\alpha ) = \sqrt{\alpha_2 +1}\,  \psi_{\alpha_1, \alpha_2+1}.  
\end{equation}

For $a_j$ we obtain 
\begin{equation}\label{a_1}
a_1(\psi_\alpha ) = \sqrt{\alpha_1}\, \psi_{\alpha_1-1,\alpha_2} + \beta_1 \sqrt{\alpha_1+1 }\, \psi_{\alpha_1+1,\alpha_2}+ \kappa \sqrt{\alpha_2 +1}\,
\psi_{\alpha_1, \alpha_2+1}
\end{equation}
and
\begin{equation}\label{a_2}
a_2(\psi_\alpha ) = \sqrt{\alpha_2}\, \psi_{\alpha_1,\alpha_2-1} + \beta_2 \sqrt{\alpha_2+1 }\, \psi_{\alpha_1,\alpha_2+1} + \kappa \sqrt{\alpha_1+1}\, \psi_{\alpha_1+1, \alpha_2}.
\end{equation}

 Now let $f\in {\text{dom}}(a_1)$ with $f= \sum_\alpha f_\alpha \psi_\alpha,$ where $(f_\alpha)_\alpha \in l^2.$ 
 
 Now we suppose that $\kappa =0.$ We use \eqref{lemm1} and get from \eqref{a_1}
\begin{equation}\label{a_f}    
 a_1(f)= \sum_\alpha \sqrt{\alpha_1+1}f_{\alpha_1+1,\alpha_2} \psi_{\alpha_1,\alpha_2} + \beta_1 \sum_{\alpha_1\ge 1, \alpha_2 \ge 0}\sqrt{\alpha_1} f_{\alpha_1-1,\alpha_2} \psi_{\alpha_1, \alpha_2}.
\end{equation}
 If $g=\sum_\alpha g_\alpha \psi_\alpha \in {\text{dom}}(a_1^*)$ we get $(a_1(f),g)=(f, a_1^*(g)).$ Comparing with the formula from above 
 \begin{equation}\label{a_*g}
 a_1^*(g)=\sum_{\alpha_1\ge 1, \alpha_2 \ge 0}
 \sqrt{\alpha_1}g_{\alpha_1-1, \alpha_2}\psi_{\alpha_1, \alpha_2}+ \beta_1 \sum_\alpha \sqrt{\alpha_1+1} g_{\alpha_1+1,\alpha_2} \psi_{\alpha_1, \alpha_2}.
 \end{equation}
If $\beta =1$ it follows that $a_1=a_1^*.$ For general $\beta_1 \in \mathbb R$ we compute the norms using Parseval's identity
\begin{eqnarray*}
\| a_1^*(f)\|^2 - \|a_1(f) \|^2 &=& \sum_{\alpha_2} |f_{1,\alpha_2}|^2 (\beta_1^2 -1)\\
&+&\sum_{\alpha_2}|f_{0, \alpha_2}+\beta_1 \sqrt{2}f_{2,\alpha_2}|^2-\sum_{\alpha_2}|\sqrt{2}f_{2, \alpha_2}+\beta_1 f_{0,\alpha_2}|^2\\
&+& \sum_{\alpha_2}|\sqrt{2}f_{1, \alpha_2}+\beta_1 \sqrt{3}f_{3,\alpha_2}|^2-\sum_{\alpha_2}|\sqrt{3}f_{3, \alpha_2}+\beta_1 \sqrt{2}f_{1,\alpha_2}|^2\\
& + & \dots \\
&=& (1-\beta_1^2) \left ( \sum_{\alpha2}|f_{0,\alpha_2}|^2 +\sum_{\alpha2}|f_{1,\alpha_2}|^2 + \dots  \right )\\
& = & (1-\beta_1^2) \|f \|^2.
\end{eqnarray*}

Hence ${\text{dom}}(a_1)={\text{dom}}(a_1^*).$

Using \eqref{b_j} we obtain
\begin{equation}\label{b_jf}
b_1(f)=\sum_{\alpha_1\ge 1, \alpha_2 \ge 0}\sqrt{\alpha_1} f_{\alpha_1-1, \alpha_2} \psi_{\alpha_1, \alpha_2} \,\,  {\text{and}} \,\,  b_1^*(f)= \sum_\alpha \sqrt{\alpha_1+1} f_{\alpha_1+1, \alpha_2}\psi_{\alpha_1, \alpha_2}.     
\end{equation}
Analogous results can be obtained for $a_2^*$ and $b_2^*.$ Putting together all these results we get the following

\begin{theorem}\label{thm1}
Let
$$A=\begin{pmatrix}\beta_1 & 0\\ 0 & \beta_2 \end{pmatrix}, $$
with $\beta_1,\beta_2 \in \mathbb R.$
Then the operators $a_j$ and $b_j$ on $H_A$ have the following properties: 
$$\| a_j^*(f)\|^2 - \|a_j(f) \|^2= (1-\beta_j^2) \|f\|^2, \ {\text{dom}}(a_j)={\text{dom}}(a_j^*),$$

$$\| b_j(f) \|^2 - \|b_j^*(f)\|^2 = \|f\|^2,  \ {\text{dom}}(b_j)={\text{dom}}(b_j^*), \, j=1,2.$$

\begin{equation}\label{dom1}
a_j^*(f) = \beta_j a_j(f)+ (1-\beta_j^2)b_j(f),
\end{equation}
for $f \in {\text{dom}}(a_j)\cap{\text{dom}}(b_j).$ 

The commutation relations are 
$$[a_j, a_j^*]=a_j a_j^* - a_j^* a_j= (1-\beta_j^2) I,$$
in addition, we have $[a_j,a_k]=0$ and
$[a_j^*,a_k^*]=0,$ for $j,k=1,2.$

If $\beta_j \neq1,$ we get the basic estimate
$$\|f\|^2 \le \frac{1}{|1-\beta_j^2|} \left (\|a_j^*(f)\|^2+\|a_j(f)\|^2 \right).$$
\end{theorem}

\begin{proof}
\eqref{dom1} follows from \eqref{a_f}, \eqref{a_*g} and \eqref{b_jf}.
To show the commutation relations we use the fact that
$a_1^* = \beta_1 a_1+ (1- \beta_1^2)b_1$ and that $[a_1,b_1]=I$ and easily 
compute $[a_1,a_1^*].$ 
\end{proof}
\begin{remark}\label{skew}
    In a similar way one can handle the case where 
    $$A=\begin{pmatrix}0 & \kappa\\ \kappa & 0 \end{pmatrix},$$  we get
$$a_1(f)-\kappa a_2^*(f) = (1-\kappa^2)b_1^*(f) \, \  {\text{and}} \,  \ a_2(f)-\kappa a_1^*(f) = (1-\kappa^2)b_2^*(f),$$
the commutation relations are 
$$[a_j, a_k^*]= \delta_{jk}(1-\kappa^2) I, \, j,k=1,2, \, $$
and $[a_j^*,a_k^*] =0, j,k=1,2.$

If  $\kappa=1$ we get $a_1(f)=a_2^*(f), $
whereas, if $\kappa=-1,$ we obtain $a_1(f)= -a_2^*(f).$
\end{remark}
\vskip 1 cm
\section{The Berezin transform}

The Berezin transform is  effective and successful in the study of Hankel and Toepltiz operators, see \cite{Zhu}.
We use the methods of chapter 2 to compute the Berezin transform of the operator $a$ on 
$$H_\beta = \{ f \ {\text{entire}} : \int_{\mathbb C} |f(z)|^2 \exp{-((1+\beta)x^2 +(1-\beta)y^2))}\, d\lambda (z) < \infty \},$$
where $0\le \beta \le 1$ and $z=x+iy.$
The functions
$$\psi_k (z) = \frac{z^k}{\sqrt{\pi k!}}\, \exp(\frac{\beta}{2}z^2), \ k=0,1, \dots $$
constitute an orthonormal basis in $H_\beta.$ Hence the  Bergman kernel $K_\beta$ of $H_\beta$ is
$$K_\beta (z,w)= \sum_{k=0}^\infty \psi_k(z) (\psi_k(w))^-= \frac{1}{\pi} \exp ( \frac{\beta}{2}(z^2+ \bar w^2))\exp (z\bar w)$$
and 
$$k_w(z) = \frac{K_\beta (z,w)}{\sqrt{K_\beta (w,w)}}.$$

We get 
$$k_w(z) = \frac{1}{\sqrt{\pi}} \, \exp ( \frac{\beta}{2}(z^2+ \frac{1}{2}\bar w^2- \frac{1}{2}w^2)+z\bar w -\frac{|w|^2}{2}) .$$
In addition, we have 
$$a(k_w)(z)=\frac{1}{\sqrt{\pi}} \, \exp ( \frac{\beta}{2}(z^2+ \frac{1}{2}\bar w^2- \frac{1}{2}w^2)+z\bar w -\frac{|w|^2}{2}) \, (\beta z+ \bar w). $$
It is easy to see that $k_w \in {\text{dom}}(a).$

The Berezin transform $\Tilde{a}$ of the operator $a$ on $H_\beta$ is given by
$$\Tilde{a}(w) =( a(k_w), k_w)_\beta$$
$$= \frac{1}{\pi}
\int_{\mathbb C} \exp ( \beta (x^2-y^2)+ 2xu+2yv-u^2-v^2)\, (\beta z +\bar w) \exp (-(1+\beta)x^2-(1-\beta)y^2)\, d\lambda (z)$$
$$=\frac{1}{\pi} \int_{\mathbb C} (\beta z +\bar w)\, \exp (-(x-u)^2-(y-v)^2)\, d\lambda (z),$$
where $z=x+iy$ and $w=u+iv).$

Finally we get
$$\Tilde a (w) = \beta w + \bar w.$$

For the adjoint operator $a^*$ one has
$$\Tilde{a^*}(w) = (a^*(k_w), k_w)_\beta = (k_w, a(k_w))_\beta =  (\Tilde{a}(w))^-.$$
Notice that the Berezin transform of $a$ on the Segal-Bargmann space ($\beta =0$) is $\Tilde{a}(w)= \bar w.$

As $a=a^*$ on $H_1,$ the Berezin transform of $a$ must be real-valued.
In fact, on $H_1$ we have $\Tilde{a}(w) = w + \bar w.$

The Berezin transform of $ab$ and $ba$ on $H_\beta$ is given by
$$\Tilde{ab}(w)= 1+ w^2 + |w|^2,
\ \Tilde{ba}(w)=w^2 + |w|^2.$$

\vskip 1 cm
The spaces $H_\beta$ can be generalized in the following way, see \cite{JPR}.

Let $\kappa >0$ and $l\in \mathbb C.$ Define
$$H_{\kappa, l}= \{ f \, {\text{entire}} \, : \| f \|_{\kappa, l}^2 = \int_{\mathbb C} |f(z)|^2\, \exp (-\kappa |z|^2 + \Re (lz^2))\, d\lambda (z) <\infty \}.$$
The functions 
$$\psi_j(z)= \frac{(\sqrt{\kappa} z)^j}{\sqrt{\pi j!}} \, e^{-lz^2/2}\ \ , \, j=0,1,\dots$$
constitute an orthonormal basis in $H_{\kappa, l}.$
The Bergman kernel of $H_{\kappa, l}$ is
$$K_{\kappa, l}(z,w) = \frac{1}{\pi}\, e^{-l(z^2+\bar w^2)/2} \, e^{\kappa z\bar w}.$$
In order to compute the Berezin transform of the operator $a$ on $H_{\kappa, l}$ we have to consider
$$k_{\kappa, l,w} (z) = \frac{K_{\kappa, l}(z,w)}{\sqrt{K_{\kappa, l}(w,w)}}.$$
We get 
$$k_{\kappa, l, w} (z)= \frac{1}{\sqrt{\pi}}\, \exp (-lz^2/2-l\bar w^2/2+ \kappa z \bar w + l w^2/4+ l\bar w^2/4- \kappa |w|^2/2),$$
as before, we have $k_{\kappa, l, w}\in {\text{dom}}(a),$ so the Berezin transform of $a$ has the form
$$\Tilde a(w) = (a(k_{\kappa, l,w}),k_{\kappa, l,w})_{\kappa, l},
$$
and an easy computation shows
$$\Tilde{a}(w)=\int_{\mathbb C} (-lz+ \kappa \bar w) \, \exp (\frac{1}{2}\Re (lw^2)-\frac{1}{2}\Re (l\bar w^2)) \exp (-\kappa (x-u)^2-
\kappa (y-v)^2)\, d\lambda (z), $$
where we set $z=x+iy$ and $w=u+iv.$ Finally we obtain
$$\Tilde{a}(w) = \bar w \exp (\frac{1}{2}\Re (lw^2)-\frac{1}{2}\Re (l\bar w^2)) - \frac{l w}{\kappa}.$$
Comparing this with the Berezin transform of $a$ in $H_\beta$ we have to set $\kappa =1 $ and $l= -\beta$ and get
$$\Tilde{a}(w) = \bar w + \beta w.$$

\section{Szeg\H{o} and Bergman kernels}

We compute the Szeg\H{o} kernel of
$$S_{\kappa, l}= \{ (\zeta_1, \zeta_2)\in \mathbb C^2 : \Im \zeta_2 = \kappa |\zeta_1|^2-\Re (l \zeta_1^2) \},$$
where $\kappa >0$ and $l \in \mathbb C.$ Hypersurfaces of this type are used as model domains for the Leray transform (
\cite{BarEdh20}). For this purpose we use the Bergman kernel $K_{\tau,\kappa, l}$ of 
$$H_{\tau,\kappa, l}
= \{ f \, {\text{entire}} \, : \| f \|_{\tau,\kappa, l}^2 = \int_{\mathbb C} |f(z)|^2\, \exp (-4\pi \tau(\kappa |z|^2 + \Re (lz^2)))\, d\lambda (z) <\infty \}.$$ 
and a formula reducing the Szeg\H{o} kernel of a hypersurface in $\mathbb C^2$ to the parameter family of Bergman kernels in spaces of entire functions in one variable, see \cite{Has94}, \cite{Has98}, \cite{Nag1}.
 
The Bergman kernel of $H_{\tau,\kappa, l}$ is
$$K_{\tau, \kappa, l}(z,w) = 4\tau \exp (-2\pi \tau l (z^2+\bar w^2)+ 4\pi \tau \kappa z \bar w).$$
Let $\mathcal D_{\kappa, l}= \{ (z_1,z_2) \in \mathbb C^2 : \Im (z_2) > \kappa |z_1|^2 + \Re (l z_1^2)\}.$
Using \cite{Has98}, the Szeg\H{o} kernel can be computed as 
$$S((z',z),(w',w))= \int_0^\infty K_{\tau, \kappa, l} (z',w') e^{-2\pi i \tau(\bar w-z)}\, d\tau,$$

which implies

$$S((z',z),(w',w))= \frac{\kappa}{\pi^2} \, ( i (\bar w-z)-2\kappa z'\bar w' + l (z'^2 + \bar w'^2))^{-2},$$
where $(z',z),(w',w)\in \mathcal D_{\kappa, l}.$


\begin{thebibliography}{99}




\bibitem{ACDKS}Alpay, D.,
 Colombo, F., Diki, K.,and Sabadini, I.
  {\it An approach to the Gaussian RBF kernels via Fock spaces}, J. Math. Phys. {\bf 63}, no. 11, 113506, 2022. 



\bibitem{BarEdh20} Barrett, D. E., and Edholm, L. D. {\it The Leray transform: Factorization, dual CR structures, and
model hypersurfaces in $\mathbb C \mathbb P^2$}, Adv. Math. {\bf 364}, 107012, 2020.
DOI :10.1016/j.aim.2020.107012






\bibitem{folland/book} Folland, G. B., {\it Harmonic analysis in phase space}, Annals of Mathematics Studies {\bf 122}, Princeton University Press, 1989.
DOI: 10.1515/9781400882427



\bibitem{Has94} Haslinger, F.,{\it Szegő kernels for certain unbounded domains in $\mathbb C^2$}, Travaux de la Conférence Internationale d'Analyse Complexe et du 7e Séminaire Roumano-Finlandais (1993). Rev. Roumaine Math. Pures Appl. {\bf 39}, no. 10, 939–950, 1994.

\bibitem{Has98} Haslinger, F.,{\it Bergman and Hardy spaces on model domains}, Illinois J. Math. {\bf 42} no. 3, 458–469, 1998.


\bibitem{hasbook} Haslinger, F., {\it The $\overline \partial$-Neumann Problem and Schr\"odinger Operators}, de Gruyter Expositions in Mathematics, vol. 59, 2nd edition, Walter de Gruyter, 2023.
DOI: 10.1515/9783111182926

\bibitem{haslinger} Haslinger, F., {\it The $\partial$-complex on the Segal-Bargmann space},  Ann. Polon. Math. {\bf 126}, 295--317, 2019.
DOI: 10.4064/ap180715-2-11

\bibitem{Has21} Haslinger, F., \textit{The generalized $\partial$-complex on the Segal-Bargmann space}, 
Operator theory, functional analysis and applications, 317–328, Oper.
Theory Adv. Appl., {\bf 282} Birkhäuser/Springer, Cham, 2021.
 
\bibitem{Has22} Haslinger, F., \textit{Basic estimates for the generalized $\partial$-complex}, Pure Appl. Math. Q., {\bf 18}, no.2, 583--597, 2022.


\bibitem{sonhas} Haslinger, F. and Son, D. N., {\it The $\partial$-complex on weighted Bergman spaces on Hermitian manifolds}, J. Math. Anal. and Appl. {\bf 487}, 123994, 2020.
DOI: 10.1016/j.jmaa.2020.123994
 
\bibitem{HS22} Haslinger, F.  and Son, D.N., \textit{The $\partial$-operator and real holomorphic vector fields}, Pure Appl. Math. Q. {\bf 18} , no. 3, 793–833, 2022.

\bibitem{JPR} Janson, S., Peetre, J. and Rochberg, R.,
{\it Hankel forms and the Fock space}, Rev. Mat. Iberoamericana, {\bf 3}, 61--138, 1987.


\bibitem{Nag1} Nagel, A., Rosay, J.-P. , Stein, E. M. and Wainger, St., {\it Estimates for the Bergman and Szegő kernels in certain weakly pseudoconvex domains}, Bull. Amer. Math. Soc. {\bf 18}, 55–59, 1988.
DOI: 10.1090/S0273-0979-1988-15598-X




\bibitem{RSN} Rouby, O., Sj\"ostrand, J. and Ngoc, S.V., {\it Analytic Bergman operators in the semiclassical limit}, Duke Math. J. {\bf 169}, 3033--3097, 2020.
DOI: 10.1215/00127094-2020-0022

\bibitem{Zhu}
Zhu, K., {\it The Berezin Transform and Its Applications}, Acta. Math. Sci. {\bf 41}, 1839–1858, 2021. DOI:org/10.1007/s10473-021-0603-5




\bibitem{Z} Zworski, M.,{\it Semiclassical analysis},
Graduate Studies in Mathematics {\bf 138}, AMS, 2012.
DOI:10.1090/gsm/138



\end{thebibliography}
\end{document}